\documentclass[a4paper]{amsart}

\usepackage{tikz,xcolor,hyperref}

\definecolor{lime}{HTML}{A6CE39}
\DeclareRobustCommand{\orcidicon}{%
	\begin{tikzpicture}
	\draw[lime, fill=lime] (0,0) 
	circle [radius=0.16] 
	node[white] {{\fontfamily{qag}\selectfont \tiny ID}};
	\draw[white, fill=white] (-0.0625,0.095) 
	circle [radius=0.007];
	\end{tikzpicture}
	\hspace{-2mm}
}

\foreach \x in {A, ..., Z}{%
	\expandafter\xdef\csname orcid\x\endcsname{\noexpand\href{https://orcid.org/\csname orcidauthor\x\endcsname}{\noexpand\orcidicon}}
}

\makeatletter
\@namedef{subjclassname@1991}{$2020$ Mathematics Subject Classification}
\makeatother
\title{Hypervaluations on hyperfields and ordered canonical hypergroups}
\author[A. Linzi \and H. Stoja\l owska]{Alessandro Linzi\orcidA{} \and Hanna Stoja\l owska}
\thanks{The authors would like to thank prof.\ F.-V.\ Kuhlmann and K.\ Kuhlmann for their careful reading and their many extremely useful remarks which helped us to improve this paper significantly.}
\address{Institute of Mathematics\\
University of Szczecin\\ Wielkopolska 15\\ 70-451 Szczecin,
Poland} \email{linzi.alessandro@gmail.com}

\address{Institute of Mathematics\\
University of Szczecin\\ Wielkopolska 15\\ 70-451 Szczecin,
Poland} \email{hanna.stojalowska@phd.usz.edu.pl}
\date{3.9.2020}
\usepackage{geometry}                
\usepackage[applemac]{inputenc}
\usepackage[T1]{fontenc}
\usepackage{indentfirst}
\usepackage{diagmac2}
\usepackage{lipsum}
\usepackage{csquotes}
\usepackage{emptypage}
\usepackage{faktor} 
\usepackage{amsmath,amsfonts,amssymb,amsthm}
\usepackage{latexsym}
\usepackage{mathdots}
\usepackage{mathrsfs}  
\newcommand{\N}{\mathbb{N}}
\newcommand{\Z}{\mathbb{Z}}
\newcommand{\Q}{\mathbb{Q}}
\newcommand{\R}{\mathbb{R}}

\theoremstyle{plain}
\newtheorem{tw}{Theorem}[section]
\newtheorem{pr}[tw]{Proposition}
\newtheorem{lem}[tw]{Lemma}

\theoremstyle{definition}
\newtheorem{df}[tw]{Definition}
\newtheorem{example}[tw]{Example}

\theoremstyle{remark}
\newtheorem{re}[tw]{Remark}
\begin{document}
\maketitle
\begin{abstract}
    We study the concept of hypervaluations on hyperfields. In particular, we show that any hypervaluation from a hyperfield onto an ordered canonical hypergroup is the composition of a hypervaluation onto an ordered abelian group (which induces the same valuation hyperring) and an order preserving homomorphism of hypergroups. 
\end{abstract}

\section{Introduction}
The theory of hyperrings and hyperfields has its origins in the paper of Krasner \cite{Krasner2}. More recently, several authors studied their theory from various points of view. For instance, J.\ Jun, in \cite{Jun}, studies algebraic geometry over hyperrings and gives the definition of hyperideals in a hyperring. M.\ Marshall in \cite{Marshall} studies the definitions of orderings and of positive cones in hyperfields. It was already mentioned in that paper that the two concepts are not equivalent as in the classical theory of those objects. In fact, we show in Example \ref{ex pos cone} that the natural generalization of the construction which in the classical case leads to this equivalence, does not work in the case of hyperfields.\par
In \cite{DeS} B.\ Davvaz and A.\ Salasi deal with hypervaluations on a hyperring onto an ordered abelian group, and also J.\ Lee in \cite{JUN} works with valued hyperfields. These two papers present different approaches in defining a hypervaluation on a hyperfield, which both appear to be interesting and well chosen. \par
In this paper we study another possibility, introduced by Kh.\ Mirdar Harijani and S.\ M.\ Anvariyeh in \emph{A hypervaluation of a hyperfield onto a totally ordered canonical hypergroup}, Studia Scientiarum Mathematicarum Hungarica 52 (1), 87-101 (2015). What they propose is a generalization of the definition of Davvaz and Salasi where the value set is allowed to be an ordered canonical hypergroup (see Definition \ref{OCH}). Here the domain is always assumed to be a hyperfield. We show in our main result, that even though this definition is proper, i.e., nontrivial examples can be found, it is not of much interest since such a hypervaluation can always be decomposed into an order preserving homomorphism of hypergroups and a hypervaluation onto an ordered abelian group (in the sense of Davvaz and Salasi) which, moreover, induces the same valuation hyperring.\par
In the paper of Mirdar Harijani and Anvariyeh, a lot of constructions are proposed without details; it turns out that they are not always possible to carry out. In the present article we also provide counterexamples to justify this last assertion.

\section{Ordered Canonical Hypergroups}

\begin{df}
Let $H$ be a nonempty set and $\mathcal{P}^* (H)$ the family of nonempty subsets of $H$. A \emph{hyperoperation} $*$ is a function which associates with every pair $(x,y) \in H \times H$ an element of $\mathcal{P}^* (H)$, denoted by $x*y$.
\end{df}

A \emph{hypergroupoid} is a nonempty set $H$ with a hyperoperation $*: H \times H \to \mathcal{P}^* (H)$. For $x \in H$, $A,B\subseteq H$ we set 
\[
A*B=\bigcup_{a\in A,b\in B} a*b,
\]
$A * x = A * \lbrace x \rbrace$ and $x *A = \lbrace x \rbrace * A$.

In 1934 the concept of a hypergroup was defined by F.\ Marty in \cite{Marty} to be a nonempty set $H$ with an associative hyperoperation (see Definition \ref{hypergp} below) such that $x * H = H * x = H$ for all $x \in H$. A special class of hypergroups, which will be of interest for us, is the following:

\begin{df} \label{hypergp}
A \emph{canonical hypergroup} is a tuple $(H,*,e)$, where $(H, *)$ is a hypergroupoid and $e$ is an element of $H$ such that the following axioms hold:
\begin{itemize}
\item[(H1)] the hyperoperation $*$ is associative, i.e., $(x*y)*z=x*(y*z)$ for all $x,y,z \in H$,
\item[(H2)] $x*y=y*x$ for all $x,y\in H$,
\item[(H3)] for every $x\in H$ there exists a unique $x'\in H$ such that $e\in x*x'$ (the element $x'$ will be denoted by $x^{-1}$),
\item[(H4)] $z\in x*y$ implies $y\in x^{-1}*z$ for all $x,y,z\in H$.
\end{itemize}
\end{df}

\begin{re}
A canonical hypergroup is a hypergroup in the sense of Marty. Fix $a \in H$ and take $x \in H*a$. Then there exist $h \in H$ such that $x \in h*a \subseteq H$, showing that $H*a \subseteq H$. For the other inclusion, take $x \in H$, then 
\[x \in x*e \subseteq x*(a^{-1}*a) = (x*a^{-1}) * a,
\]
so there exist $h \in x*a^{-1} \subseteq H$ such that $x \in h*a \subseteq H*a$.
\end{re}

\begin{re}
In \cite{Mittas} the definition of a canonical hypergroup also requires explicitly that $x*e = \lbrace x \rbrace$ for all $x \in H$. However, we note that this axiom follows from (H3) and (H4). Indeed, suppose that $y \in x*e$ for some $x,y \in H$. Then $e \in x^{-1} * y$ by (H4). Now $y = x$ follows from the uniqueness required in (H3).
\end{re}

\begin{re} \label{grhy}
Note that an abelian group $G$ is not a priori a hypergroup, because the operation on $G$ is not a hyperoperation, as it takes values in $G$ and not in $\mathcal P^* (G)$. But it can be turned into a hypergroup by setting $a * b := \lbrace ab \rbrace$. In other words, we can turn an abelian group into a hypergroup by identifying each element $a$ of $G$ with the singleton $\lbrace a \rbrace$.
\end{re}
\vspace{0,2cm}

\begin{example} \label{sign hgp}
Consider the set $H: =\lbrace -1, 0 , 1 \rbrace$ with a hyperoperation $*$ defined as follows:
\begin{align*}
(-1) * (-1) &=(-1)*0= 0*(-1)=\{-1\}\\
0 * 0 &= \{0\}\\
1 * 1 &=1*0=0*1= \{1\}\\
1 * (-1) &= (-1)*1 =\lbrace -1, 0 , 1\rbrace. \end{align*}
Then $(H, *, 0)$ is a canonical hypergroup, called the \emph{sign hypergroup}. As the reader may check, we have that $1^{-1} = -1$, $(-1)^{-1} = 1$, $0^{-1} = 0$.
\end{example}

The next example can be found in \cite{Krasner}.

\begin{example}\label{quotient}
Let $R$ be a ring and $G$ a normal subgroup of its multiplicative semigroup. Consider the following equivalence relation $\sim$ on $R$:
$a \sim b$ if and only if there exist $g,h \in G$ s.t. $ag = bh$. The equivalence class of $a \in R$ is 
\[ aG:= \lbrace ag \mid g \in G \rbrace.\] 
It is possible to define a hyperoperation on $R/G$ in the following way:
\[
aG + bG := \lbrace (ag+bh)G \mid g,h \in G \rbrace.
\]
Then $(R/G,+,\lbrace 0_R \rbrace)$ is a canonical hypergroup. Indeed, the associative law follows from the same law in $R$, as well as commutativity. 
The unique inverse of $aG$ is $(-a)G$. Indeed, 
\[
aG + (-a)G = \lbrace (ag-ah)G \mid g,h \in G \rbrace \ni (a-a)G = 0_R G,
\]
moreover, if
\[
0_R G \in aG + bG = \lbrace (ag+bh)G \mid g,h \in G \rbrace,
\]
then there exist $g, h \in G$ such that $0_R = ag+bh$. Multiplying by $g^{-1}$, we obtain that $-a = bhg^{-1}$ and so $(-a)G = bG$.

Assume now that $cG \in aG+bG$. We wish to show that $bG \in (-a)G + cG$. We have 
\[ 
cG \in \lbrace (ag+bh)G \mid g,h \in G \rbrace,
\]
so there exist $g,h \in G$ such that $c = ag+bh$. Multiplying by $h^{-1}$, we obtain $b = -agh^{-1} + ch^{-1}$, so $bG \in \lbrace ((-a)g' + ch')G \mid g', h' \in G \rbrace = (-a)G+cG$.
\end{example}

\begin{re} \label{re}
We note that the sign hypergroup can be obtained as a quotient in the way described in the previous example. Take $R=\R$ and $G=\dot \R^2$, where $\dot\R^2$ denotes the set of non-zero squares in $\R$. The result follows from the fact that every non-zero real number is either a square or the opposite of a square. 
\end{re}

In their paper Mirdar Harijani and Anvariyeh consider the set $\mathbb Z / \mathbb N = \lbrace a \mathbb N \mid a \in \mathbb Z \rbrace$ with the hyperoperation defined as follows:
\[a \mathbb N * b\mathbb N = \lbrace c \mathbb N \mid c \in a \mathbb N + b\mathbb N \rbrace.\]
This construction does not give a canonical hypergroup in the sense of Definition \ref{hypergp}, as is shown in the next example. The reason is that $\mathbb N$ is not a multiplicative subgroup of the multiplicative semigroup of $\mathbb Z$.

\begin{example}
Consider the tuple $(\Z / \N, *, 0 \N)$. We observe that $0 \N \in  1 \mathbb N * -k \mathbb N$ for every $k \in \N$. Hence axiom (H3) is not fulfilled.
\end{example}

\begin{df}\label{OCH}
An \emph{ordered canonical hypergroup} is a tuple $(H, *, e, \leq )$, where $(H, *, e)$ is a canonical hypergroup and $\leq$ is a partial order such that 
\begin{equation} \label{ord}
    a \leq b \Longrightarrow a*c \nearrow b*c
\end{equation} 
for all $a,b,c, \in H$; here, if $A, B \subseteq H$, then $A \nearrow B$ means that for all $b \in B$ there exists $a \in A$ such that $a \leq b$. If for each $a, b \in H$ either $a \leq b$ or $b \leq a$, then we call $\leq$ an \emph{ordering} or a \emph{linear order relation} on $H$.
\end{df}

\begin{re}
One should not use \lq\lq$\leq$\rq\rq\ on the subsets as the relation will not be antisymmetric. For example if $H$ contains a smallest element $b$, then any two subsets $A,B$, which both contain $b$ will satisfy $A \nearrow B$ and $B \nearrow A$ without necessarily being equal.
\end{re}

\begin{re}
It should be noted that some authors in defining $A \nearrow B$ require that for all $a \in A$ there exists $b \in B$ and for all $b \in B$ there exists $a \in A$ such that $a \leq b$. Others instead require only that for all $a \in A$ there exists $b \in B$ such that $a \leq b$. The relations between all these definitions still have to be investigated. However, in what follows we will just use the concept introduced in Definition \ref{OCH}. 
\end{re}

\begin{lem} \label{FVK}
Let $(H, *, e, \leq)$ be an ordered canonical hypergroup. Take $a,b,x,y \in H$ and $B \subseteq H$.
\begin{enumerate}
    \item[1)] If $\lbrace a \rbrace \nearrow B$, then $a \leq b$ for all $b \in B$.
    \item[2)] If $x > e$, then $x^{-1} < e$.
    \item[3)] If $x \geq e$ and $y \geq e$, then $b \geq e$ for all $b \in x*y$.
    \item[4)] If $x > e$ and $y \geq e$, then $b > e$ for all $b \in x*y$.
\end{enumerate}
\end{lem}

\begin{proof}
1): This follows from the definition of ``$\nearrow$'' and the fact that $\{a\}$ contains only $a$.
\\

2): By condition (\ref{ord}), $\{x^{-1}\}=0*x^{-1}\nearrow x*x^{-1}$, so $x^{-1}\leq b$ for every $b\in x*x^{-1}$
by part 1). As $e\in x*x^{-1}$, we obtain that 
$x^{-1}\leq e$. Now $x^{-1}=e$ is impossible because otherwise, $e\in x*x^{-1}=\{x\}$, which implies that $x=e$ 
in contradiction to our assumption that $x>e$.
\\ 

3): If $x\geq e$ and $y\geq e$, then $\{y\}=e*y\nearrow x*y$, so $e\leq y\leq b$ for every $b\in x*y$
by part 1).
\\

4): If $x> e$ and $y\geq e$, then $b\geq e$ for all $b\in x*y$ by part 3). However, $e\notin x*y$ since otherwise,
$y=x^{-1}$ which by part 2) yields that $y<e$, in contradiction to our assumption that $y\geq e$. Hence $b\ne e$ and consequently, $b>e$.
\end{proof}

In \cite{Marshall} the definition of a positive cone is given for hyperfields. Let us now define this concept for canonical hypergroups.
\begin{df}
A subset $P$ of a canonical hypergroup $(H, *, e)$ is called a \emph{positive cone} if the following axioms hold:
\begin{itemize}
\item[(P1)] $P \cap -P = \lbrace e \rbrace$,
\item[(P2)] $P * P \subseteq P$,
\item[(P3)] $P\cup -P=H$.
\end{itemize}
\end{df}

In the theory of ordered abelian groups, the existence of a positive cone is equivalent to the existence of an ordering and there is a one to one correspondence between them. It was already mentioned in \cite{Marshall} that this, in general, is no more true in the case of hyperfields, and the argument holds for hypergroups as well. Mirdar Harijani and Anvariyeh claim that one can always construct an ordering from a positive cone by setting:
\begin{equation}\label{2}
x \leq y \iff (y * x^{-1}) \cap P \neq \emptyset.
\end{equation}
In the following example we show that this construction is not always possible.

\begin{example} \label{ex pos cone}
Take $H:= \mathbb Q / \dot{\mathbb Q}^{2}$, where $\dot{\mathbb Q}^{2}$ denotes the set of nonzero squares in $\Q$. We see that $H$ is a canonical hypergroup with the hyperoperation defined as:
\[
a \dot {\mathbb Q}^{2} * b \dot{\mathbb Q}^{2} = \lbrace c\dot{\mathbb Q}^{2} \mid c \in a\dot{\mathbb Q}^{2} + b\dot{\mathbb Q}^{2} \rbrace
\] 
(see also Example \ref{quotient}). Observe that the set $P = {\mathbb Q}^{+} / \dot{\mathbb Q}^{2} \cup \lbrace 0 \rbrace$ fulfils the conditions of the definition of a positive cone. However, the relation defined in (\ref{2}) is not an ordering on $H$. Indeed, observe that $ 5 \dot{\mathbb Q}^{2} \in (2 \dot{\mathbb Q}^{2} * (3 \dot{\mathbb Q}^{2})^{-1}) \cap P$, so, in particular, $(2 \dot{\mathbb Q}^{*2} * (3 \dot{\mathbb Q}^{2})^{-1}) \cap P \neq \emptyset$, which means that $3 \dot{\mathbb Q}^{2} \leq 2 \dot{\mathbb Q}^{2}$. On the other hand, $1 \dot{\mathbb Q}^{2} \in (3 \dot{\mathbb Q}^{2} * (2 \dot{\mathbb Q}^{2})^{-1}) \cap P$, so $(3 \dot{\mathbb Q}^{2} * (2 \dot{\mathbb Q}^{2})^{-1}) \cap P \neq \emptyset$, which means that $2 \dot{\mathbb Q}^{2} \leq 3 \dot{\mathbb Q}^{2}$. Clearly, $2 \dot{\mathbb Q}^{2} \neq 3 \dot{\mathbb Q}^{2}$, so the relation $\leq$ is not antisymmetric.
\end{example}

However, there exist hypergroups in which positive cone and orderings behave as in the classical case. The simplest example is the sign hypergroup:

\begin{example} 
Consider the sign hypergroup $H = \lbrace -1, 0 ,1 \rbrace$ with positive cone $P = \lbrace 0, 1 \rbrace$. We define an order relation $\leq$ on $H$ as follows: $-1 \leq 0 \leq 1$. We leave it to the reader to show that $\leq$ is a linear order on $H$ as in Definition \ref{OCH}. This ordering corresponds to $P$ in the way described in (\ref{2}). Indeed, $(1 * 0^{-1}) \cap P = \lbrace 1 \rbrace$, $(0 * (-1)^{-1}) \cap P = \lbrace 1 \rbrace$, $(1 * (-1)^{-1}) \cap P = \lbrace 0, 1 \rbrace$, so in every case we obtain a nonempty intersection.
\end{example}

\section{Hyperrings and hyperfields}

\begin{df}
A \emph{hyperring} is a tuple $(R,+,\cdot,0,1)$ which satisfies the following axioms:
\begin{itemize}
\item[(R1)] $(R,+,0)$ is a canonical hypergroup,
\item[(R2)] $(R,\cdot,1)$ is a commutaive monoid such that $x\cdot0=0$ for all $x\in R$,
\item[(R3)] the operation $\cdot$ is distributive with respect to the hyperoperation $+$. That is, for all $x,y,z\in R$,
\[
x(y+z)=xy+xz
\]
as sets. Here for $x\in H$ and $A\subseteq H$ we have set
\[
xA:=\{xa\mid a\in A\}.
\]
\end{itemize}
If $(R\setminus\{0\}, \cdot,1)$ is an abelian group, then $(R,+,\cdot,0,1)$ is called \emph{hyperfield}.
\end{df}

The following example is the original example of a quotient hyperring (see \cite{Krasner}).

\begin{example} \label{quot}
Let $R$ be a commutative ring with $1$, $G$ a normal subgroup of its multiplicative semigroup and recall the notations introduced in Example \ref{quotient}. 
One may define multiplication in $R/G$ as $aG \cdot bG:= (ab)G$. Then $(R/G, +, \cdot, \lbrace 0 \rbrace, 1_RG)$ is a hyperring and if $R$ is a field, then it is a hyperfield. We have to show that (R3) holds. Take $a,b,c \in R$ and suppose that $xG \in aG \cdot (bG + cG)$. Then there exist $g,h \in G$ such that $x = a(bg+ch) = abg + ach$, where we have used the distributivity law in $R$. We obtain 
\[
xG \in \lbrace (abg + ach)G \mid g,h \in G \rbrace = abG + acG = aG bG + aG cG.
\]
This shows that $aG \cdot (bG+cG) \subseteq aG bG + aG cG$. To show the converse inclusion we note that
\[
xG \in aG bG + aG cG = (ab)G + (ac)G.
\]
Hence there exist $g,h \in G$ such that $x = (ab)g + (ac)h = a(bg+ch)$, where we used the distributivity law in $R$. We obtain 
\[
xG \in \lbrace a(bg+ch)G \mid g,h \in G \rbrace = \lbrace aG(bg + ch)G \mid g,h \in G \rbrace = aG \cdot (bG + cG).
\]
If $R$ is a field, then for every $a \in R$ we have that $a^{-1}G = (aG)^{-1}$. Indeed, by definition $aG \cdot a^{-1}G = (a \cdot a^{-1})G = 1_RG$.
\end{example}

\begin{example}
As we have already noted in Remark \ref{re} that the sign hypergroup $H$ can be seen as a quotient $\R / \dot\R^2$. Since $\R$ is a field, we obtain that $(H, +, \cdot, 0, 1)$ is a hyperfield, where we now denote $*$ by $+$ and $\cdot$ behaves as follows: 
\begin{align*}
    -1 \cdot 1 &= 1 \cdot -1 = -1,\\
    0 \cdot 0 &= 0 \cdot 1 = 1 \cdot 0 = 0 \cdot -1 = -1 \cdot 0 = 0, \\
    1 \cdot 1 &= -1 \cdot -1 = 1.
\end{align*}
\end{example}
\vspace{0,2cm}
In their paper Mirdar Harijani and Anvariyeh, in an attempt to generalize the construction of formal power series, argue that the set $X$ of order preserving mappings $f:H\to K$ such that the support of $f$ is finite, where $H$ is a ordered canonical hypergroup and $K$ an ordered hyperfield, is a hyperdomain, i.e., a hyperring without zero divisors, with the operations defined as follows:
\[
(f+g)(x):=f(x)+_K g(x),
\]
and
\begin{equation}\label{1}
(f g)(x):=\sum_{x\in x_1* x_2}f(x_1)g(x_2).
\end{equation}
First of all we note that $X$ is not even a hyperring since if $f$ is non-constant and order preserving, then $-f$ is not order preserving. In addition, if $K$ is a hyperfield which is not a field, then the multiplication $(fg)(x)$ defined in (\ref{1}) is a subset of $K$ and not an element of $K$ as the definition of hyperring would require. 
Hence we do not obtain a hyperring. Finally, even if we do not restrict to order preserving mappings and assume $K$ to be a field, we show, in the next example, that the multiplication defined in (\ref{1}) is not associative.

\begin{example}
Let $H=\{-1,0,1\}$ denote the sign hypergroup and $K$ an ordered field. We define the maps $f, g, h:H\to K$ as follows $f(-1)=f(0) = f(1)=1_K$, $g(-1) = g(0)=g(1)=-1_K$, and $h(1)=1_K$, $h(0)=-1_K$, $h(-1)=0_K$. By direct computations we obtain
\begin{align*}
((fg)h)(1) &= 5 \cdot (-1_K) + 5 \cdot 1_K + 3 \cdot (-1_K) + 5 \cdot (-1_K) = 8 \cdot (-1_K) 
\end{align*}
and
\begin{align*}
(f(gh))(1) &= 2 \cdot (-1_K) + 2 \cdot (-1_K) + 2 \cdot (-1_K) = 6 \cdot (-1_K)
\end{align*}
where we used the fact that $1\in 1* 1,1* 0,0* 1,1*-1,-1* 1$, that $-1\in -1* -1,-1* 0,0* -1,1* -1,-1 * 1$ and that $0\in0* 0,1* -1,-1* 1$. Here $*$ denotes the operation of $H$. We conclude that $f(gh)\neq (fg)h$, hence the associativity law does not hold.
\end{example}
\section{Hypervaluations}
As it was mentioned before, there are several approaches to the definition of a hypervaluation on a hyperfield. We now wish to investigate the following.

\begin{df}\label{hyperval}
Let $(F,+,\cdot,0,1)$ be a hyperfield and $(H,*,e,\leq)$ be an ordered canonical hypergroup. A surjective map $w: F \to H \cup \lbrace \infty \rbrace$   is called a \emph{hypervaluation} on the hyperfield $F$ if the following properties are satisfied:
\begin{itemize}
    \item [(V1)] $w(x) = \infty \iff x = 0$,
    \item [(V2)] $w(-x) = w(x)$,
    \item [(V3)] $w( x \cdot y) \in w(x) * w(y)$,
    \item [(V4)] $z \in x+y \implies w(z) \geq \min \lbrace w(x), w(y) \rbrace$.
    \end{itemize}
\end{df}

The following is a nontrivial example of a hypervaluation onto an ordered canonical hypergroup. 

\begin{example} \label{ex hypval}
Let $K$ be a field and $\Gamma$ an ordered abelian group. Assume that a classical valuation 
\[
v:K\to\Gamma\cup\{\infty\}
\]
is given. We now define a hypervaluation $w$ from $K$ onto the sign hypergroup $\{-1,0,1\}$. 
\[
w(x)=\begin{cases}1&\text{if }\infty\neq v(x)>0_\Gamma\\
0&\text{if }v(x)=0_\Gamma\\
-1&\text{if }v(x)<0_\Gamma\\
\infty&\text{otherwise}
\end{cases}
\]
Let us show that $w$ is a hypervaluation on the field $K$, as in Definition 4.1. Clearly, $w(x)=\infty$ if and only if $x=0_K$ and $w(x)=w(-x)$ because that is true for $v$. In order to show that $w(xy)\in w(x) * w(y)$, where $*$ denotes the operation in the sign hypergroup, we observe that if $v(x)$ and $v(y)$ have the same sign there is nothing to show, since $v(xy)=v(x)+v(y)$ will have the same sign as $v(x)$ and $v(y)$. 
If $v(x)=0$, then $w(x)=0$ and $v(xy)=v(y)$ and $w(xy) = w(y) = 0 * w(y)$; the same holds if $v(y) = 0$. If $x=0$ or $y=0$, then the situation is clear.
If, say, $v(x)<0_\Gamma$ and $v(y)>0_\Gamma$, then $w(xy)\in \{-1,0,1\} = -1 * 1 = w(x) * w(y)$, and similarly in the case where $v(x)>0_\Gamma$ and $v(y)<0_\Gamma$. This shows that the third axiom of a hypervaluation holds for $w$. The fourth and last axiom states:
\[
z\in x+y \Longrightarrow w(z)\succeq\min\{w(x),w(y)\}
\]
where $\preceq$ denotes the ordering in the sign hypergroup.
Since $K$ is a field we have to check that $w(x+y)\succeq\min\{w(x),w(y)\}$. 
If $w(x) \neq w(y)$, then clearly $v(x) \neq v(y)$, so we have $v(x+y) = \min \lbrace v(x), v(y) \rbrace$. Hence $w(x+y) = \min \lbrace w(x), w(y) \rbrace$. If $x=y=0$ there is nothing to show. If $w(x)=w(y)=0$, then $v(x)=v(y)=0_\Gamma$, so $v(x+y) \geq 0_\Gamma$. Then $w(x+y) \in \lbrace 0,1, \infty \rbrace$ and $w(x+y)\succeq 0 = \min\{w(x),w(y)\}$. If $w(x)=w(y)=1$, then $v(x), v(y) > 0_\Gamma$, so $v(x+y) > 0_\Gamma$. We obtain $w(x+y) \in \lbrace 1, \infty \rbrace$ and $w(x+y)\succeq 1 = \min\{w(x),w(y)\}$. Finally, if $w(x)=w(y)=-1$, then $w(x+y) \in \lbrace -1, 0, 1, \infty \rbrace$, but $\min \lbrace w(x), w(y) \rbrace = -1$, so the fourth axiom also holds for $w$.
\end{example}

\begin{lem} \label{valpro}
Let $w: F \to H \cup \lbrace \infty \rbrace$ be a hypervaluation. Then $w(1_F) = e$ and $w(x^{-1}) = (w(x))^{-1}$.
\end{lem}

\begin{proof}
Let $x \in F$ be such that $w(x) = e$. Then $e = w(x \cdot 1_F) \in w(x) * w(1_F) = \lbrace w(1_F) \rbrace$, hence $e = w(1_F)$. To prove the second assertion observe that, since $xx^{-1} = 1_F$, we have that $w(xx^{-1}) = w(1_F) = e$. By axiom (V3) we obtain that $e=w(xx^{-1}) \in w(x) * w(x^{-1})$, then axiom (H3) implies that $w(x^{-1})=(w(x))^{-1}$.
\end{proof}

\begin{df}
Let $(R, +, \cdot, 0,1)$ be a hyperring. An element $x \in R$ is called a \emph{unit of $R$} if there exists $y \in R$ such that $x \cdot y = 1$.
\end{df}

\begin{df}
Let $F$ be a hyperfield and $R \subseteq F$ a hyperring with respect to the hyperaddition and multiplication of $F$. If for every $x \in F$ either $x \in R$ or $x^{-1} \in R$, then $R$ is called a \emph{valuation hyperring}.
\end{df}

\begin{df}
Let $R$ be a hyperring. 
\begin{enumerate}
\item A nonempty subset $I \subseteq R$ is a \emph{hyperideal} if for all $a,b \in I$ and for all $r \in R$ we have $a+b \subseteq I$, $-a \in I$ and $ar \in I$.
\item A hyperideal $I \subsetneq R$ is \emph{maximal} if $I$ satisfies the following property: if $J \subseteq R$ is a hyperideal of $R$ such that $I\subsetneq J$, then $J = R$.
\end{enumerate}
\end{df}

In the paper of Mirdar Harijani and Anvariyeh one can find the first three statements of the following proposition. For the sake of completeness we rewrite the proof and complete it with details where needed. For our purposes, we also add a fourth statement.

\begin{pr} \label{prop}
Let $(F, +, \cdot, 0, 1)$ be a hyperfield, $(H, *,e,\leq)$ an ordered canonical hypergroup and $w: F \rightarrow H\cup\{\infty\}$ a hypervaluation. 
\begin{enumerate}
\item[1)] The set $\mathcal{O}_w = \lbrace x \in F \mid w(x) \geq e \rbrace$ is a valuation hyperring.
\item[2)] The set $U_w = \lbrace x \in F \mid w(x) = e \rbrace$ is a group under the multiplication of the hyperfield $F$ and consists of all units of $\mathcal O_w$.
\item[3)] The set $\mathfrak{m}_w = \lbrace x \in F \mid w(x) > e \rbrace$ is the unique maximal hyperideal of $\mathcal O_w$.
\item[4)] The quotient $G:= (F\setminus\{0\}) / U_w$ is an ordered abelian group with respect to the operation $xU_w \cdot yU_w = xyU_w$ and the ordering $xU_w \leq yU_w \Leftrightarrow yx^{-1} \in \mathcal O_w$.
\end{enumerate}
\end{pr}

\begin{proof}
1) If $x, y \in \mathcal O_w$, then $w(x), w(y) \geq e$, so by part 3) of Lemma \ref{FVK} we obtain that for all $b \in w(x) * w(y)$ we have $b \geq e$. Since $w(xy) \in w(x) * w(y)$, in particular $w(xy) \geq e$, hence $xy \in \mathcal O_w$.
The inclusion $x+y \subseteq \mathcal O_w$ follows from (V4) and from the fact that $\min \lbrace w(x), w(y) \rbrace \geq e$. Indeed by (V4), 
\[
z\in x+y \Longrightarrow w(z) \geq \min\{w(x),w(y)\} \geq e.
\]
To see that $\mathcal O_w$ is a valuation hyperring, take $x \in F \setminus \mathcal O_w$. Then $w(x) <e$, which implies that $w(x^{-1}) = (w(x))^{-1} >e$ by Lemma \ref{valpro} and part 2) of Lemma \ref{FVK}. Thus, $x^{-1} \in \mathcal O_w$.\\

2) If $x \in U_w$, then $x^{-1} \in U_w$ too. Indeed, $w(x^{-1}) = (w(x))^{-1} = e$. Hence the elements of $U_w$ are units of $\mathcal O_w$. To show the other inclusion, assume that $x \in \mathcal O_w \setminus U_w$. Then $w(x)>e$, so $w(x^{-1})<e$ by part 2) of Lemma \ref{FVK}, which means that $x^{-1} \notin \mathcal O_w$.\\

3) Let $x,y \in \mathfrak{m}_w$. Then by (V4) we obtain that 
\[
z\in x+y \Longrightarrow w(z)\geq\min\{w(x),w(y)\} > e.
\]
Thus, $x+y \subseteq \mathfrak{m}_w$. From $w(-x)=w(x)$ it follows that $-x\in\mathfrak{m}_w$. If $x \in \mathfrak{m}_w$ and $y \in \mathcal O_w$, then $w(x) > e$ and $w(y) \geq e$, so by part 4) of Lemma \ref{FVK} we obtain that for every $b \in w(x)*w(y)$ we have $b>e$. Since by (V3) $w(xy) \in w(x) * w(y)$, we conclude that $xy \in \mathfrak {m}_w$. We have proved that $\mathfrak {m}_w$ is a hyperideal. It follows from part 2) and from the fact that if a hyperideal contains a unit, then is not proper, that $\mathfrak m_w$ is the unique maximal hyperideal of $\mathcal O_w$.\\

4)  By part 2) $U_w$ is a subgroup of the abelian group $(F \setminus \lbrace 0 \rbrace, \cdot, 1)$. So $G$ is the quotient group with respect to this (normal) subgroup. The proof that $\leq$ is a linear order on $G$ is exactly the same as in the classical case.\qedhere
\end{proof}

\begin{df}
Let $(H_1,*_1,e_1)$ and $(H_2,*_2,e_2)$ be canonical hypergroups. 
\begin{enumerate}
\item  A \emph{homomorphism of hypergroups} is a function $f:H_1\to H_2$ such that $f(e_1)=e_2$ and $f(a*_1b)\subseteq f(a)*_2f(b)$.
\item A \emph{strong homomorphism of hypergroups} is a function $f:H_1\to H_2$ such that $f(e_1)=e_2$ and $f(a*_1b)= f(a)*_2f(b)$ as sets.
\item An \emph{isomorphism of hypergroups} is a strong homomorphism of hypergroups which is bijective.
\end{enumerate}
\end{df}

In their paper, Mirdar Harijani and Anvariyeh claim that two hypervaluations $v: F \to H_v\cup\{\infty\}$ and $w: F \to H_w\cup\{\infty\}$ induce the same valuation hyperring if and only if there exists an order preserving isomorphism $f$ between the value hypergroups $H_v, H_w$ such that $w = f \circ v$. Although this statement is true in classical valuation theory, this is no longer true for hypervaluations.

\begin{example}
Consider the hypervaluations $v, w$ from Example \ref{ex hypval}. We observe that the valuation ring $\mathcal O_v$ of $v$ coincides with the valuation (hyper)ring of $w$. Indeed, if $v(x)\geq 0_\Gamma$, then $w(x)\in\{0,1\}$, so $w(x)\succeq 0$ which means that $\mathcal O_v\subseteq\mathcal O_w$. On the other hand, if $w(x)\succeq 0$, then, by definition, $v(x)\geq 0$, so $x\in \mathcal O_w$ implies $x\in\mathcal O_v$. Note that there is no order preserving isomorphism between the ordered abelian group $\Gamma$ and the sign hypergroup. But in the present case, as shown above, they define the same valuation (hyper)ring in $K$.
\end{example}

The next theorem states that any hypervaluation from a hyperfield onto an ordered canonical hypergroup is the composition of a hypervaluation onto an ordered abelian group (which induces the same valuation hyperring) and an order preserving homomorphism of hypergroups.

\begin{tw} \label{main}
Let $F$ be a hyperfield, $(H, *,e, \leq)$ an ordered canonical hypergroup and $w:F \rightarrow H \cup \lbrace \infty \rbrace$ a hypervaluation. Then there exists a hypervaluation $v:F \rightarrow G \cup \lbrace \infty \rbrace$, where $G$ is an ordered abelian group, and an order preserving homomorphism $h:G \rightarrow H$ of hypergroups from $G$ to $H$ s.t. $w = h \circ v$ and $\mathcal O_v = \mathcal O_w$.
\end{tw}

\begin{proof}
By part 4) of Proposition \ref{prop} we can consider the ordered abelian group $G= (F\setminus\{0\}) / U_w$. Let $h : G \rightarrow H $ be the map 
\[
h(xU_{w}) = w(x).
\]
We first show that $h$ is well defined. Suppose that $xU_w = yU_w$. This is equivalent to $xy^{-1} \in U_w$, which holds if and only if $w(xy^{-1}) = e$. From (V3) and Lemma \ref{valpro} we obtain that $e \in w(x) * w(y^{-1}) = w(x) * (w(y))^{-1}$ and this means that $w(x) = w(y)$ by the uniqueness required in axiom (H3). This proves that $h$ is well defined.
To show that $h$ is a homomorphism of hypergroups, note that $h(1U_w) = w(1) = e$ by Lemma \ref{valpro}. Take $xU_w, yU_w \in G$. According to Remark \ref{grhy} we obtain that
\[
h(xU_w \cdot yU_w) = \lbrace h(xyU_w) \rbrace = \lbrace w(xy) \rbrace \subseteq w(x) * w(y) = h(xU_w) * h(yU_w),
\]
which proves that $h$ is a homomorphism of hypergroups. Moreover, assume that $xU_w \leq yU_w$, which by definition is equivalent to $e \leq w(yx^{-1})$. From the latter it follows by condition (\ref{ord}) of Definition \ref{OCH} that $\lbrace w(x) \rbrace = e * w(x) \nearrow w(yx^{-1}) * w(x)$. Hence by part 1) of Lemma \ref{FVK}, $w(x) \leq b$ for every $b \in w(yx^{-1}) * w(x)$. As $w(y) = w(yx^{-1}x) \in w(yx^{-1}) * w(x)$, we obtain that $w(x) \leq w(y)$. This proves that $h$ is order preserving.\par
We now define $v: F \rightarrow G \cup \lbrace \infty \rbrace$ as follows:
\[
v(x) = \left\{ \begin{array}{ll}
x U_w,  & \textnormal{ if $x \neq 0$}\\
\infty, & \textnormal{ if $ x=0 $}
\end{array} \right.
\]
The map $v$ is a hypervaluation onto the ordered abelian group $G$ (again, modulo the provision of Remark \ref{grhy}).
The only axiom which needs justification is the fourth: if $z \in x+y$ for some $x,y \in F$, then we wish to show that $v(z) \geq \min \lbrace v(x), v(y) \rbrace$. Assume, without loss of generality, that $v(x) \leq v(y)$, so that $yx^{-1} \in \mathcal O_w$. From $z \in x+y$ it follows that 
\[
yx^{-1} \in (x-z)x^{-1} = 1 - zx^{-1}.
\]
This means that $zx^{-1} \in yx^{-1} - 1 \subseteq \mathcal O_w,$ which proves that $v(z) \geq v(x)$.

Finally, one can see that $w = h \circ v$. Indeed, if $x \in F$, then if $x \neq 0$, we have that $h(v(x)) = h(xU_w) = w(x)$ by definition. If $x= 0$, the situation is clear once we set $h(\infty):=\infty$. It is left to show that $\mathcal O_v = \mathcal O_w$. We observe that $x \in \mathcal O_w$ if and only if $1U_w \leq xU_w$, which happens if and only if $x \in \mathcal O_v$. This completes the proof.
\end{proof}

\end{document}